\documentclass[12pt]{amsart}

\newtheorem{theorem}{Theorem}[section]
\newtheorem{lemma}[theorem]{Lemma}

\newtheorem{cor}[theorem]{Corollary}

\theoremstyle{definition}

\theoremstyle{remark}
\newtheorem{remark}[theorem]{Remark}

\numberwithin{equation}{section}

\newcommand{\rr}{{\mathbb R}}
\newcommand{\nat}{{\mathbb N}}
\newcommand{\ganz}{{\mathbb Z}}

\newcommand{\Var}{\operatorname{Var}}

\newcommand{\eqd}{\stackrel{\rm d}{=}}
\newcommand{\Exp}{\mathbb E}

\begin{document}
\sloppy
\title[On some compound distributions with Borel summands]{On some compound distributions\\ with Borel summands}

\author{H. Finner}
\address{Helmut Finner, Institute of Biometrics and Epidemiology, German Diabetes Center at the Heinrich-Heine-Universit\"at D\"usseldorf, Auf'm Hennekamp 65, D-40225 D\"usseldorf, Germany} \email{finner\@@{}ddz.uni-duesseldorf.de}

\author{P. Kern}
\address{Peter Kern, Mathematical Institute, Heinrich-Heine-Universit\"at D\"usseldorf, Universit\"atsstr.\ 1, D-40225 D\"usseldorf, Germany} \email{kern\@@{}hhu.de}

\author{M. Scheer}
\address{Marsel Scheer, Institute of Biometrics and Epidemiology, German Diabetes Center at the Heinrich-Heine-Universit\"at D\"usseldorf, Auf'm Hennekamp 65, D-40225 D\"usseldorf, Germany} \email{m.kvesic\@@{}web.de}

\date{\today}

\begin{abstract}
The generalized Poisson distribution is well known to be a compound Poisson distribution with Borel summands. As a generalization we present closed formulas for compound Bartlett and Delaporte distributions with Borel summands and a recursive structure for certain compound shifted Delaporte mixtures with Borel summands. Our models are introduced in an actuarial context as claim number distributions and are derived only with probabilistic arguments and elementary combinatorial identities. In the actuarial context related compound distributions are of importance as models for the total size of insurance claims for which we present simple recursion formulas of Panjer type.
\end{abstract}



\maketitle

\allowdisplaybreaks

\section{Introduction}

A random variable $Z$ is said to have a compound distribution if it is of the form 
$$Z\eqd\sum_{k=1}^NY_k,$$
where $(Y_n)_{n\in\nat}$ is an i.i.d.\ sequence and $N$ is an independent random variable with values in $\nat_0$. Throughout this paper ``$\eqd$'' denotes equality in distribution and the empty sum $\sum_{k=1}^0$ is taken to be zero. The most prominent example is a compound Poisson distribution, where $N$ has a Poisson distribution, which is closely related to infinite divisibility. It is well known that a discrete random variable on $\nat_0$ is infinitely divisible if and only if it has a compound Poisson distribution; e.g.\ see page 290 in \cite{Fel1}. We will focus on compound distributions, where the i.i.d.\ summands $(Y_n)_{n\in\nat}$ have a Borel distribution but the distribution of $N$ can be more general than Poisson. 

In fact this note is inspired by some asymptotic distribution results in multiple hypotheses testing. The distribution of the number of false rejections in so-called linear step-down (SD) and step-up (SU) procedures under a certain Dirac-uniform configuration obtained by Dempster \cite{Dem}, respectively Finner and Roters \cite{FR}, have recently been shown by Scheer \cite{Sche} to have the following asymptotic distributions in case the number of hypotheses increases to infinity but the (unknown) number of false hypotheses is kept fixed. These asymptotic distributions are
\begin{equation}\label{SD}
p_{S\!D}(n)=\frac{\theta(\theta+\lambda n)^{n-1}}{n!}\,e^{-(\theta+\lambda n)}\quad\text{ for }n\in\nat_0
\end{equation}
and
\begin{equation}\label{SU}
p_{SU}(n)=\frac{(1-\lambda)(\theta+\lambda n)^{n}}{n!}\,e^{-(\theta+\lambda n)}\quad\text{ for }n\in\nat_0,
\end{equation}
for certain parameters $\theta>0$ and $\lambda\in(0,1)$; cf.\ also Theorem 4.1 in \cite{FS}. The first is well known as a generalized Poisson distribution (GPD) which can be represented as a compound Poisson distribution with Borel summands. The second is known as a linear function Poisson distribution by Jain \cite{Jain}, appearing as a weighted Langrangian distribution in \cite{Jan}. In our context it will turn out to be a compound Bartlett distribution with Borel summands. Since the Bartlett distribution is the convolution of a Poisson and a geometric distribution, it is a special case of the Delaporte distribution, which is the convolution of a Poisson and a negative binomial distribution. Hence we will also ask for the generalization of compound Delaporte distributions with Borel summands. We will further analyze the natural generalizations with $k\in\ganz$
$$p_k(n)=C\,\frac{(\theta+\lambda n)^{n+k-1}}{n!}\,e^{-(\theta+\lambda n)}\quad\text{ for }n\in\nat_0,$$
where $C=C(k,\theta,\lambda)$ is some normalization constant. For $k\in\nat$ these will turn out to be certain compound shifted Delaporte mixtures with Borel summands.

The afore mentioned distributions (Poisson, Bartlett, Delaporte) are frequently used in actuarial mathematics to model the total number of insurance claims under various conditions. Hence, as in case of the GPD, corresponding compound distributions with Borel summands are of considerable interest in asymptotic statistics as well as in actuarial modeling.

In Section 2 we review on known facts concerning the Borel and Borel-Tanner distribution. These are derived as limit distributions of total progeny in a certain branching process of Galton-Watson type with deterministic initial value. We introduce this model in an actuarial context for which the limit of total progeny is interpreted as a total number of insurance claims. The limit distribution leads to a well known functional equation for the generating function of a Borel distribution which is commonly solved by Lagrange's inversion in the literature. In the Appendix we prefer to derive the Borel and Borel-Tanner distribution from the functional equation using only probabilistic arguments and elementary combinatorial identities. In Section 3 we will consider the model of Section 2 with random initial values which leads to compound distributions with Borel summands. Starting with a compound Poisson distribution, well known as the generalized Poisson distribution, we will derive the natural generalizations of compound Bartlett and Delaporte distributions and certain compound shifted Delaporte mixtures. An overview of the presented compound distributions with Borel summands is given in Section 3.5, in which we also point out the connection to Lagrangian probability distributions. Since the considered compound distributions fulfill recursive relations, we will show in Section 4 how to recursively evaluate related compound sum distributions with a Panjer type formula.

\section{The Borel distribution}

Following section 2.7 in Consul \cite{Con}, a Borel distribution can be introduced as the limit distribution of total progeny in a certain branching process of Galton-Watson type. We give an actuarial interpretation of this model.

Assume $X_0$ is a (random) number of initial insurance claims. Each of these claims is supposed to induce independently a random number of secondary consequential i.i.d.\ claims and so on. Since the number of secondary claims is likely to be large only with very small probability, the number of consequential claims is assumed to follow a Poisson distribution. Hence let $(Y_{k,n})_{k,n\in\nat}$ be an i.i.d.\ array of random variables with a Poisson distribution of parameter $\lambda>0$, where $Y_{k,n}$ represents the number of secondary claims resulting from the $k$-th claim in the $(n-1)$-th step. Then the total number of claims up to the $m$-th step is given by
\begin{equation}\label{totnum}
Z_m=X_0+\cdots+X_m,\quad\text{ where }\quad X_n=\sum_{k=1}^{X_{n-1}}Y_{k,n}\quad\text{ for }n=1,\ldots,m.
\end{equation}
Let $g(z)=\exp(\lambda(z-1))$ be the probability generating function (pgf) of each $Y_{k,n}$ and let $G_m$ denote the pgf of $Z_m$.

Now assume $X_0=1$, then the pgf of $Z_1=X_0+X_1=1+Y_{1,1}$ is $G_1(z)=z\cdot g(z)$. Since each of the $Y_{1,1}$ claims in the first step will independently generate a total number of secondary claims up to the $m$-th step which is distributed as $Z_{m-1}$, we get $G_m(z)=G_1(G_{m-1}(z))$, or inductively $G_m(z)=G_1^{\circ m}(z)$ for every $m\in\nat$. In case the Galton-Watson branching process in \eqref{totnum} will eventually stop, i.e.
\begin{equation}\label{stop}
P\{X_n=0\text{ for some }n\in\nat\}=1,
\end{equation}
the limit $G(z)=\lim_{m\to\infty}G_m(z)$ exists and is the pgf of some random variable $Y$ with values in $\nat$. Since $P\{Y_{1,1}=0\}=e^{-\lambda}>0$, it is well known that \eqref{stop} holds if and only if $\Exp[Y_{1,1}]=\lambda\leq 1$, e.g.\ see section 2.2 in \cite{Bre}. Since $G_1(z)=z\cdot g(z)$, we obtain for the limiting pgf
\begin{equation}\label{Borelpgf}
G(z)=z\cdot g(G(z))=z\cdot\exp\big(\lambda(G(z)-1)\big)
\end{equation}
and it is well known that the unique solution of \eqref{Borelpgf} is the pgf $G$ of a Borel distribution. For $\lambda\in(0,1]$ the distribution of the limiting random variable $Y$ is given by the Borel distribution of parameter $\lambda$ with
\begin{equation*}
P\{Y=n\}=\frac{(\lambda n)^{n-1}}{n!}\,e^{-\lambda n}\quad\text{ for }n\in\nat.
\end{equation*}
A common way to show this is using Lagrange's inversion formula, e.g.\ see P\'olya and Szeg\H{o} \cite{PS}, page 145 and Hurwitz and Courant \cite{HC}, page 135 for a proof of Lagrange's formula using tools from complex analysis. We rather prefer to give a direct proof in the Appendix using only probabilistic arguments and elementary combinatorial identities. 

The functional equation \eqref{Borelpgf} for the pgf of a Borel distribution enables to calculate cumulants of the Borel distribution by differentiation; see section 7.2.2 in \cite{JKK}. From the first two cumulants we get the well known expectation and variance of a Borel distribution. For $\lambda\in(0,1)$ we have
\begin{equation}\label{borelmom}
\Exp[Y]=\frac1{1-\lambda}\quad\text{ and }\quad\Var(Y)=\frac{\lambda}{(1-\lambda)^3}
\end{equation}
and for $\lambda=1$ the expectation $\Exp[Y]$ does not exist.\\

Now assume $X_0=m$ for the number of initial insurance claims. Since each of the initial $m$ claims independently generates a Borel distributed number of claims in total, the limit distribution of $Z_n$ as $n\to\infty$ is an $m$-fold convolution of the Borel distribution known as the Borel-Tanner distribution, e.g.\ see \cite{HB}.
For fixed $m\in\nat$ let $Y^{(m)}=Y_1+\cdots+Y_m$, where $Y_1,\ldots,Y_m$ are i.i.d.\ random variables with a Borel distribution of parameter $\lambda\in(0,1]$. Then $Y^{(m)}$ has a Borel-Tanner distribution with
\begin{equation}\label{BorelTandense}
P\{Y^{(m)}=n\}=\frac{m\,(\lambda n)^{n-m}}{n\,(n-m)!}\,e^{-\lambda n}\quad\text{ for }n\in\nat\text{ with }n\geq m.
\end{equation}
We will derive the Borel-Tanner distribution in the Appendix using our elementary approach.
For $\lambda\in(0,1)$ we immediately deduce the expectation and variance of a Borel-Tanner distribution from \eqref{borelmom}
\begin{equation}\label{BorelTanmom}
\Exp[Y^{(m)}]=\frac{m}{1-\lambda}\quad\text{ and }\quad\Var(Y^{(m)})=\frac{m\lambda}{(1-\lambda)^3}.
\end{equation}

\section{Compound distributions with Borel summands}

We will further follow the branching process approach to the Borel distribution given in Section 2 in its actuarial interpretation. Instead of a deterministic number of initial insurance claims, we will now consider $X_0$ to be random with certain claim number distributions widely applied in the actuarial literature.

\subsection{Compound Poisson distribution, the generalized Poisson distribution}

Assume $X_0$ has a Poisson distribution of parameter $\theta>0$, then the limit distribution of $Z_n$ as $n\to\infty$ is a compound Poisson distribution with Borel summands also known as the generalized Poisson distribution (GPD); e.g.\ see section 2.7 in Consul \cite{Con}. Note that by Theorem 2.1 of Joe and Zhu \cite{JZ} or by Theorem 5.1 of Pakes \cite{Pak} it is also possible to represent the GPD as a certain variance mixture of Poisson distributions. 
\begin{theorem}\label{gpd}
Let $(Y_n)_{n\in \nat}$ be an i.i.d.\ sequence with a Borel distribution of parameter $\lambda\in(0,1]$ and let $N$ be an independent random variable with a Poisson distribution of parameter $\theta>0$. Then $Z=\sum_{k=1}^NY_k$ has a GPD with
\begin{equation}\label{gpddens}
P\{Z=n\}=\frac{\theta(\theta+\lambda n)^{n-1}}{n!}\,e^{-(\theta+\lambda n)}\quad\text{ for }n\in\nat_0.
\end{equation}
\end{theorem}
\begin{proof}
For $n=0$ we get $P\{Z=0\}=P\{N=0\}=e^{-\theta}$ in agreement with \eqref{gpddens} and for $n\in\nat$ we obtain from \eqref{BorelTandense}
\begin{align*}
P\{Z=n\} & =P\left\{\sum_{k=1}^NY_k=n\right\}=\sum_{m=1}^\infty P\left(\sum_{k=1}^mY_k=n\bigg| N=m\right)\cdot P\{N=m\}\\
& =\sum_{m=1}^n\frac{\theta^m}{m!}\,e^{-\theta}P\{Y^{(m)}=n\}=\sum_{m=1}^n\frac{\theta^m}{m!}\,e^{-\theta}\,\frac{m(\lambda n)^{n-m}}{n\,(n-m)!}\,e^{-\lambda n}\\
& =\frac{\theta}{n!}\left(\sum_{m=1}^n\frac{(n-1)!}{(m-1)!\,(n-m)!}\,\theta^{m-1}(\lambda n)^{n-m}\right)e^{-(\theta+\lambda n)}\\
& =\frac{\theta}{n!}\left(\sum_{m=0}^{n-1}{n-1\choose m}\,\theta^{m}(\lambda n)^{n-1-m}\right)e^{-(\theta+\lambda n)}\\
& =\frac{\theta(\theta+\lambda n)^{n-1}}{n!}\,e^{-(\theta+\lambda n)}
\end{align*}
concluding the proof.
\end{proof}
We can directly calculate the expectation and variance of a GPD from \eqref{borelmom} and Wald's identities. For $\lambda=1$ the expectation $\Exp[Z]$ does not exist and for $\lambda\in(0,1)$ we have
\begin{align*}
& \Exp[Z]=\Exp[N]\cdot\Exp[Y_1]=\frac{\theta}{1-\lambda}
\intertext{and}
& \Var(Z)=\Exp[N]\cdot\Var(Y_1)+\Var(N)\cdot\Exp[Y_1]^2=\frac{\theta\lambda}{(1-\lambda)^3}+\frac{\theta}{(1-\lambda)^2}=\frac{\theta}{(1-\lambda)^3}.
\end{align*}
\begin{remark}
It is possible to extend the parameters of a GPD. For $\lambda=0$ we see from \eqref{gpddens} that $Z$ has a Poisson distribution of parameter $\theta$, since the Borel distributed random variables fulfill $P\{Y_k=1\}=1$. In case $\theta=0$ we get $P\{Z=0\}=1$. Lerner, Lone and Rao \cite{LLR} have shown that \eqref{gpddens} defines a sub-probability measure for $\lambda>1$ and calculated its total mass. Hence an appropriate normalization again leads to a probability measure. This reflects the fact that in case $\lambda>1$ the Galton-Watson branching process does not stop with positive probability. Consul and Jain \cite{CJ} state that also in case $\lambda\in(-1,0)$ a distribution is defined by truncating \eqref{gpddens} only to those $n\in\nat_0$ for which $\theta+\lambda n>0$. This appears not to be true as can easily be derived for the case $\theta=1$, $\lambda=-1/2$ but, of course, again an appropriate normalization leads to a probability distribution. However, for our considerations negative values of $\lambda$ are of no interest.
\end{remark}
For parameters $\theta>0$ and $\lambda\in(0,1)$ let $p(\theta,\lambda;n)=P\{Z=n\}$, $n\in\nat_0$, be the GPD given by \eqref{gpddens}. It is easy to see that the GPD fulfills the recursive relation
\begin{equation}\label{gpdrecur}
p(\theta,\lambda;n)=\frac{\theta}{\theta+\lambda}\left(\lambda+\frac{\theta}{n}\right)\,p(\theta+\lambda,\lambda;n-1)\quad\text{ for }n\in\nat,
\end{equation}
cf.\ also equation (4.4) in \cite{AB}. The GPD in \eqref{gpddens} coincides with the distribution $p_{S\!D}$ in \eqref{SD}. We will now show that \eqref{SU} is a compound Bartlett distribution with Borel summands.

\subsection{Compound Bartlett distribution}

To see that \eqref{SU} in fact defines a probability distribution, we obtain using \eqref{SD} and the above expectation of a GPD
\begin{align*}
\sum_{n=0}^\infty p_{SU}(n) & =\sum_{n=0}^\infty\frac{(1-\lambda)(\theta+\lambda n)}{\theta}\,p_{S\!D}(n)\\
& =(1-\lambda)\sum_{n=0}^\infty p_{S\!D}(n)+\frac{\lambda(1-\lambda)}{\theta}\sum_{n=0}^\infty n\cdot p_{S\!D}(n)\\
& =1-\lambda+\frac{\lambda(1-\lambda)}{\theta}\,\frac{\theta}{1-\lambda}=1.
\end{align*}
Now let $M$ be a geometrically distributed random variable on $\nat_0$ with $P\{M=n\}=\lambda^n(1-\lambda)$ for $n\in\nat_0$ and some parameter $\lambda\in(0,1)$. Then $M$ has pgf $H(z)=\frac{1-\lambda}{1-\lambda z}$, expectation $\frac{\lambda}{1-\lambda}$ and variance $\frac{\lambda}{(1-\lambda)^2}$. The distribution of the sum of $M$ with an independent Poisson random variable of parameter $\theta>0$ is known as the Bartlett distribution due to its appearance in \cite{Bar}. The probabilities of a Bartlett distributed random variable $N$ can only be given in form of a convolution
\begin{equation}\label{Barconv}
P\{N=n\}=\sum_{k=0}^n P\{M=n-k\}\,\frac{\theta^{k}}{k!}\,e^{-\theta}=(1-\lambda)\lambda^n e^{-\theta}\sum_{k=0}^n \frac1{k!}\,\left(\frac{\theta}{\lambda}\right)^k
\end{equation}
for $n\in\nat_0$. We will now determine the distribution of a compound Bartlett distribution with Borel summands of the same parameter $\lambda\in(0,1)$.
\begin{theorem}\label{Barmix}
Let $(Y_n)_{n\in \nat}$ be an i.i.d.\ sequence with a Borel distribution of parameter $\lambda\in(0,1)$ and let $N$ be an independent random variable with a Bartlett distribution given by \eqref{Barconv} with additional parameter $\theta\geq0$. Then $Z=\sum_{k=1}^NY_k$ has the distribution given by \eqref{SU}
\begin{equation}\label{Barmixdens}
P\{Z=n\}=\frac{(1-\lambda)(\theta+\lambda n)^{n}}{n!}\,e^{-(\theta+\lambda n)}\quad\text{ for }n\in\nat_0.
\end{equation}
\end{theorem}
\begin{proof}
We first consider $\theta>0$. Let $G$ be the pgf of a Borel distribution and $H$ be the pgf of a geometric distribution both with the same parameter $\lambda\in(0,1)$. Further let $G_{S\!D}$ and $G_{SU}$ denote the pgf's of the distributions given by \eqref{SD} and \eqref{SU}, respectively. Since \eqref{SD} defines a GPD we know that $G_{S\!D}(z)=\exp(\theta(G(z)-1))$ and hence by differentiation we get 
$$G_{S\!D}'(z)=\theta G'(z)G_{S\!D}(z).$$
We further obtain $z=G(z)\exp(-\lambda(G(z)-1))$ from \eqref{Borelpgf} which by differentiation yields $1=G'(z)(1-\lambda G(z))\exp(-\lambda(G(z)-1))$ or
\begin{equation}\label{Gprime}
G'(z)=\frac{\exp(\lambda(G(z)-1))}{1-\lambda G(z)}=\frac{G(z)}{z(1-\lambda G(z))}.
\end{equation}
Since $p_{SU}(n)=(1-\lambda)p_{S\!D}(n)+\lambda(1-\lambda)\theta^{-1}\,n\, p_{S\!D}(n)$, the above derivatives lead to
\begin{align*}
G_{SU}(z) & =(1-\lambda)\sum_{n=0}^\infty z^n p_{S\!D}(n)+\frac{\lambda(1-\lambda)z}{\theta}\sum_{n=0}^\infty n\,z^{n-1} p_{S\!D}(n)\\
& =(1-\lambda)G_{S\!D}(z)+\frac{\lambda(1-\lambda)z}{\theta}\,G_{S\!D}'(z)\\
& =(1-\lambda)\big(1+\lambda z G'(z)\big)G_{S\!D}(z)=(1-\lambda)\left(1+\frac{\lambda G(z)}{1-\lambda G(z)}\right)G_{S\!D}(z)\\
& = \frac{1-\lambda}{1-\lambda G(z)}\,G_{S\!D}(z)=H(G(z))\exp(\theta(G(z)-1)),
\end{align*}
which is the pgf of the proposed compound Bartlett distribution with Borel summands.

For $\theta=0$ note that the distribution of $Z$ is a compound geometric distribution with Borel summands and is the weak limit as $\theta\downarrow0$ of the above compound Bartlett distribution. Hence \eqref{Barmixdens} is also valid in case $\theta=0$.
\end{proof}
We can again directly calculate the expectation and variance of the compound Bartlett distribution from \eqref{borelmom} and Wald's identities
\begin{align*}
& \Exp[Z]=\Exp[N]\cdot\Exp[Y_1]=\left(\theta+\frac{\lambda}{1-\lambda}\right)\frac1{1-\lambda}=\frac{\theta}{1-\lambda}+\frac{\lambda}{(1-\lambda)^2}
\intertext{and}
& \Var(Z)=\Exp[N]\cdot\Var(Y_1)+\Var(N)\cdot\Exp[Y_1]^2=\frac{\theta}{(1-\lambda)^3}+\frac{\lambda^2+\lambda}{(1-\lambda)^4}.
\end{align*}
For parameters $\theta\geq0$ and $\lambda\in(0,1)$ let $p(\theta,\lambda;n)=P\{Z=n\}$, $n\in\nat_0$, be the compound Bartlett distribution given by \eqref{Barmixdens}, which coincides with \eqref{SU}. Again, it is easy to see that the compound Bartlett distributions fulfill the recursive relation
\begin{equation}\label{Barmixrecur}
p(\theta,\lambda;n)=\left(\lambda+\frac{\theta}{n}\right)\,p(\theta+\lambda,\lambda;n-1)\quad\text{ for }n\in\nat.
\end{equation}

\subsection{Compound Delaporte distribution}

Let $M^{(m)}$ be a random variable with a negative binomial distribution $P\{M^{(m)}=n\}={n+m-1\choose n}\lambda^n(1-\lambda)^m$ for $\lambda\in(0,1)$ and $m\in\nat$. Since the negative binomial distribution is an $m$-fold convolution of the geometric distribution, $M^{(m)}$ has pgf $H^m(z)=\left(\frac{1-\lambda}{1-\lambda z}\right)^m$, expectation $\frac{m\lambda}{1-\lambda}$ and variance $\frac{m\lambda}{(1-\lambda)^2}$. The distribution of the sum of $M^{(m)}$ with an independent Poisson random variable of parameter $\theta>0$ is known as the Delaporte distribution although its first appearance goes back to L\"uders \cite{Lue}; see section 5.12.5 in \cite{JKK}. The probabilities of a Delaporte distributed random variable $N$ can only be given in form of a convolution
\begin{equation}\label{Delconv}\begin{split}
P\{N=n\} &=\sum_{k=0}^n P\{M^{(m)}=k\}\,\frac{\theta^{n-k}}{(n-k)!}\,e^{-\theta}\\
& =\sum_{k=0}^n {k+m-1\choose k}\lambda^k(1-\lambda)^m\,\frac{\theta^{n-k}}{(n-k)!}\,e^{-\theta}\\
& =\frac{(1-\lambda)^me^{-\theta}}{(m-1)!\,n!}\sum_{k=0}^n {n\choose k} (k+m-1)!\,\lambda^k\theta^{n-k}
\end{split}\end{equation}
for $n\in\nat_0$. We will now determine the distribution of a compound Delapote distribution with Borel summands of  the same parameter $\lambda\in(0,1)$. This distribution generalizes an $\alpha$-modified Poisson distribution of type $m-1$ given by Chakraborty \cite{Cha}; see also Steliga and Szynal \cite{SS}.
\begin{theorem}\label{Delmix}
Let $(Y_n)_{n\in \nat}$ be an i.i.d.\ sequence with a Borel distribution of parameter $\lambda\in(0,1)$ and let $N$ be an independent random variable with a Delaporte distribution given by \eqref{Delconv} with additional parameters $\theta\geq0$ and $m\geq2$. Then $Z^{(m)}=\sum_{k=1}^NY_k$ has the distribution 
\begin{equation}\label{Delmixdens}
P\{Z^{(m)}=n\}=\frac{(1-\lambda)^m(\theta+\lambda n+\lambda\,\alpha(m-1))^{n}}{n!}\,e^{-(\theta+\lambda n)}\quad\text{ for }n\in\nat_0,
\end{equation}
where we use Riordan's \cite{Rio} $\alpha$-symbols defined by $\alpha^\ell(m-1)={m+\ell-2\choose \ell}\ell!$ when applying the binomial formula to the factor $(\theta+\lambda n+\lambda\,\alpha(m-1))^{n}$ in \eqref{Delmixdens} for $\ell=0,\ldots, n$.
\end{theorem}
\begin{proof}
We first consider $\theta>0$. Let $G$ be the pgf of a Borel distribution and $H$ be the pgf of a geometric distribution both with the same parameter $\lambda\in(0,1)$. Then the proposed compound Delaporte distribution has pgf 
$$H^m(G(z))\exp(\theta(G(z)-1))=\Big(H(G(z))\exp\left(\tfrac{\theta}{m}(G(z)-1)\right)\Big)^m$$
and thus is an $m$-fold convolution of the compound Bartlett distribution in Section 3.2, where the parameter $\theta$ has to be replaced by $\theta/m$. Hence write $Z^{(m)}=\sum_{k=1}^mZ_k$, where $Z_1,\ldots,Z_m$ are i.i.d.\ with probability mass function
$$P\{Z_1=n\}=\frac{(1-\lambda)\big(\frac{\theta}{m}+\lambda n\big)^{n}}{n!}\,e^{-\left(\frac{\theta}{m}+\lambda n\right)}.$$
For the compound Delaporte distribution we obtain
\begin{align*}
P\{Z^{(m)}=n\} & =P\left\{\sum_{k=1}^mZ_k=n\right\}=\sum_{(n_1,\ldots,n_m)\in\nat_0^m\atop n_1+\cdots+n_m=n}\prod_{k=1}^mP\{Z_1=n_k\}\\
& =\sum_{(n_1,\ldots,n_m)\in\nat_0^m\atop n_1+\cdots+n_m=n}\prod_{k=1}^m\frac{(1-\lambda)\big(\frac{\theta}{m}+\lambda n_k\big)^{n_k}}{n_k!}\,e^{-\left(\frac{\theta}{m}+\lambda n_k\right)}\\
& =\frac{(1-\lambda)^m \lambda^n e^{-(\theta+\lambda n)}}{n!}\sum_{(n_1,\ldots,n_m)\in\nat_0^m\atop n_1+\cdots+n_m=n}\frac{n!}{n_1!\cdots n_m!}\prod_{k=1}^m\left(\frac{\theta}{m\lambda}+n_k\right)^{n_k}.
\end{align*}
The sum on the right-hand side is a multinomial Abel sum $A_{n}\big(\frac{\theta}{m\lambda},\ldots,\frac{\theta}{m\lambda},0,\ldots,0\big)$ with $m$ factors as defined in section 1.6 of \cite{Rio}. Its closed-form solution due to Hurwitz is given by 
$$\sum_{(n_1,\ldots,n_m)\in\nat_0^m\atop n_1+\cdots+n_m=n}\frac{n!}{n_1!\cdots n_m!}\prod_{k=1}^m\left(\frac{\theta}{m\lambda}+n_k\right)^{n_k}=\left(\frac{\theta}{\lambda}+n+\alpha(m-1)\right)^n$$
with $\alpha(m-1)$ as in the statement of Theorem \ref{Delmix}; cf.\ \cite{Rio}, page 25.

For $\theta=0$ note that the distribution of $Z$ is a compound negative binomial distribution with Borel summands and is the weak limit as $\theta\downarrow0$ of the above compound Delaporte distribution. Hence \eqref{Delmixdens} is also valid in case $\theta=0$.
\end{proof}
We can again directly calculate the expectation and variance of the compound Delaporte distribution from \eqref{borelmom} and Wald's identities
\begin{align*}
& \Exp[Z^{(m)}]=\Exp[N]\cdot\Exp[Y_1]=\left(\theta+\frac{m\lambda}{1-\lambda}\right)\frac1{1-\lambda}=\frac{\theta}{1-\lambda}+\frac{m\lambda}{(1-\lambda)^2}
\intertext{and}
& \Var(Z^{(m)})=\Exp[N]\cdot\Var(Y_1)+\Var(N)\cdot\Exp[Y_1]^2=\frac{\theta}{(1-\lambda)^3}+\frac{m(\lambda^2+\lambda)}{(1-\lambda)^4}.
\end{align*}
\begin{lemma} \label{Lem:Delmixrec}
For $\theta\geq0$, $\lambda\in(0,1)$ and $m\geq2$ let $p(\theta,\lambda,m;n)=P\{Z^{(m)}=n\}$, $n\in\nat_0$, be the compound Delaporte distribution given by \eqref{Delmixdens}. Then for every $n\in\nat$ the compound Delaporte distributions fulfill the recursive relation
\begin{equation}\label{Delmixrecur}
p(\theta,\lambda,m;n)=\frac{\lambda (m-1)}{(1-\lambda)n}\,p(\theta+\lambda,\lambda,m+1;n-1)+\frac{\theta + \lambda n}{n}\,p(\theta+\lambda,\lambda,m;n-1).
\end{equation}
\end{lemma}
\begin{proof}
First note that
\begin{align*}
\alpha^{n-k}(m-1) &= {m+n-k-2 \choose n-k} (n-k)! = \frac{(m+n-k-2)!}{(n-k)! (m-2)!} (n-k)! \\
&= (m-1) \frac{(m+1 + (n-1-k) -2)!}{(n-1-k)!(m+1-2)!} (n-1-k)! \\
&= (m-1) \alpha^{n-1-k}(m)
\end{align*}
for $k = 0, \ldots, n-1$.
Now, from Theorem \ref{Delmix} we obtain
\begin{align*}
& p(\theta,\lambda,m;n)=\frac{(1-\lambda)^m(\theta+\lambda n+\lambda\,\alpha(m-1))^{n}}{n!}\,e^{-(\theta+\lambda n)}\\ 
& \quad=\frac{(1-\lambda)^m}{n!}\,e^{-(\theta+\lambda n)}\sum_{k=0}^n{n\choose k}(\theta+\lambda n)^k\lambda^{n-k}\alpha^{n-k}(m-1)\\
& \quad=\frac{(1-\lambda)^m}{n!}\,e^{-(\theta+\lambda n)}\left(\sum_{k=0}^{n-1}{n-1\choose k}(\theta+\lambda n)^k\lambda^{n-k}\alpha^{n-k}(m-1)\right.\\
& \qquad\phantom{=\frac{(1-\lambda)^m}{n!}\,e^{-(\theta+\lambda n)}}\left.+\sum_{k=1}^n{n-1\choose k-1}(\theta+\lambda n)^k\lambda^{n-k}\alpha^{n-k}(m-1)\right)\\
& \quad=\frac{(1-\lambda)^m}{n!}\,e^{-(\theta+\lambda n)}\left((m-1)\lambda\sum_{k=0}^{n-1}{n-1\choose k}(\theta+\lambda n)^k\lambda^{n-1-k}\alpha^{n-1-k}(m)\right.\\
& \qquad\phantom{=\frac{(1-\lambda)^m}{n!}\,e^{-(\theta+\lambda n)}}\left.+(\theta+\lambda n)\sum_{k=0}^{n-1}{n-1\choose k}(\theta+\lambda n)^k\lambda^{n-1-k}\alpha^{n-1-k}(m-1)\right)\\
& \quad=\frac{(1-\lambda)^m}{n\,(n-1)!}\,e^{-(\theta+\lambda n)}\left((m-1)\lambda\cdot(\theta+\lambda n+\lambda\,\alpha(m))^{n-1} \right. \\
& \quad\phantom{=\frac{(1-\lambda)^m}{n\,(n-1)!}\,e^{-(\theta+\lambda n)}} \left. +(\theta+\lambda n)(\theta+\lambda n+\lambda\,\alpha(m-1))^{n-1}\right)\\
& \quad=\frac{\lambda (m-1)}{(1-\lambda)n}\,p(\theta+\lambda,\lambda,m+1;n-1)+\frac{\theta+\lambda n}{n}\,p(\theta+\lambda,\lambda,m;n-1),
\end{align*}
which proves the assertion.
\end{proof}

\subsection{Compound shifted Delaporte mixtures}

We will now consider the following natural generalization of the GPD \eqref{SD} and the compound Bartlett distribution \eqref{SU}
\begin{equation}\label{rsDm}
p_k(\theta,\lambda;n)=\frac{1}{S(k,\theta,\lambda)}\,\frac{(\theta+\lambda n)^{n+k-1}}{n!}\,e^{-(\theta+\lambda n)},\quad{n\in\nat_0}
\end{equation}
for $k\in\ganz$, $\theta\geq0$ and $\lambda\in(0,1)$ with the obvious restriction $\theta>0$ in case $k\leq0$ and with normalizing constants introduced by Consul and Jain \cite{CJ}
$$S(k,\theta,\lambda)=\sum_{n=0}^\infty\frac{(\theta+\lambda n)^{n+k-1}}{n!}\,e^{-(\theta+\lambda n)}.$$
For $k=0$ and $\theta>0$ Consul and Jain recover $p_0(\theta,\lambda;n)=p_{S\!D}(n)$ with $S(0,\theta,\lambda)=\theta^{-1}$ and for $k=1$ we see that $p_1(\theta,\lambda;n)=p_{SU}(n)$ with $S(1,\theta,\lambda)=(1-\lambda)^{-1}$. The motivation in \cite{CJ} for introducing these quantities is to calculate the mean and variance of a GPD by means of the recursive relation
\begin{equation*}\begin{split}
S(k,\theta,\lambda) & =\theta\sum_{n=0}^\infty\frac{(\theta+\lambda n)^{n+(k-1)-1}}{n!}\,e^{-(\theta+\lambda n)}+\lambda\sum_{n=1}^\infty\frac{(\theta+\lambda n)^{(n-1)+k-1}}{(n-1)!}\,e^{-(\theta+\lambda n)}\\
&=\theta S(k-1,\theta,\lambda)+\lambda S(k,\theta+\lambda,\lambda).
\end{split}\end{equation*}
Starting with $S(0,\theta,\lambda)=\theta^{-1}$ (respectively $S(1,\theta,\lambda)=(1-\lambda)^{-1}$ in case $\theta=0$) a repeated use of this formula enables to calculate the normalizing constants recursively by
\begin{equation}\label{Srec}
S(k,\theta,\lambda)=\begin{cases}
\displaystyle\sum_{n=0}^\infty\lambda^n(\theta+\lambda n)\,S(k-1,\theta+\lambda n,\lambda) & \quad\text{ for }k\in\nat,\\
\theta^{-1}\big(S(k+1,\theta,\lambda)-\lambda S(k+1,\theta+\lambda,\lambda)\big) & \quad\text{ for }k\in-\nat.\end{cases}
\end{equation}
For example we get
$$S(-1,\theta,\lambda)=\frac1{\theta}\left(\frac1{\theta}-\frac{\lambda}{\theta+\lambda}\right)=\frac{\theta(1-\lambda)+\lambda}{\theta^2(\theta+\lambda)}>0.$$
The approach of Consul and Jain generalizes to all higher order moments of the distributions \eqref{rsDm} as follows.
\begin{lemma}\label{mom}
For $k\in\ganz$, $\theta\geq0$ and $\lambda\in(0,1)$, with $\theta>0$ in case $k\leq0$, let $X_k(\theta,\lambda)$ be a random variable on $\nat_0$ with distribution \eqref{rsDm}. Then for $m\in\nat$ we have
\begin{equation}\label{rsDmmom}
\Exp\left[(X_k(\theta,\lambda))^m\right]=\frac{S(k+1,\theta+\lambda,\lambda)}{S(k,\theta,\lambda)}\sum_{\ell=0}^{m-1}{m-1\choose\ell}\Exp\left[(X_{k+1}(\theta+\lambda,\lambda))^\ell\right].
\end{equation}
The obvious relation $\Exp[(X_k(\theta,\lambda))^0]=1$ together with the recursion \eqref{Srec}, or \eqref{Salt} below, enables to calculate the moments \eqref{rsDmmom} explicitly.
\end{lemma}
\begin{proof}
By \eqref{rsDm} we have for $m\in\nat$
\begin{align*}
\Exp\left[(X_k(\theta,\lambda))^m\right] & =\sum_{n=0}^\infty n^mp_k(\theta,\lambda;n)=\frac{1}{S(k,\theta,\lambda)}\sum_{n=0}^\infty n^m\frac{(\theta+\lambda n)^{n+k-1}}{n!}\,e^{-(\theta+\lambda n)}\\
& =\frac{1}{S(k,\theta,\lambda)}\sum_{n=0}^\infty (n+1)^{m-1}\frac{((\theta+\lambda)+\lambda n)^{n+(k+1)-1}}{n!}\,e^{-((\theta+\lambda)+\lambda n)}\\
& =\frac{1}{S(k,\theta,\lambda)}\sum_{n=0}^\infty\sum_{\ell=0}^{m-1}{m-1\choose\ell} n^\ell S(k+1,\theta+\lambda,\lambda)\,p_{k+1}(\theta+\lambda,\lambda;n)\\
& =\frac{S(k+1,\theta+\lambda,\lambda)}{S(k,\theta,\lambda)}\sum_{\ell=0}^{m-1}{m-1\choose\ell}\Exp\left[(X_{k+1}(\theta+\lambda,\lambda))^\ell\right]
\end{align*}
concluding the proof.
\end{proof}
It is easy to see that for $k\in\ganz$ the distributions \eqref{rsDm} fulfill the recursive relation
\begin{equation}\label{rsDmrecur}
p_k(\theta,\lambda;n)=\frac{S(k,\theta+\lambda,\lambda)}{S(k,\theta,\lambda)}\,\left(\lambda+\frac{\theta}{n}\right)\,p_k(\theta+\lambda,\lambda;n-1)\quad\text{ for }n\in\nat.
\end{equation}
\begin{remark}
Alternatively, it is also possible to calculate the moments of $X_k(\theta,\lambda)$ in Lemma \ref{mom} using the identity
$$\Exp\left[\big(\theta+\lambda\cdot X_k(\theta,\lambda)\big)^m\right]=\frac{S(k+m,\theta,\lambda)}{S(k,\theta,\lambda)},$$
which directly follows from \eqref{rsDm}.
\end{remark}
Our aim is to show that \eqref{rsDm} defines a certain compound distribution with Borel summands at least for $k\in\nat_0$. Since for $k=0$ \eqref{rsDm} defines a GPD it suffices to consider $k\in\nat$.
\begin{theorem}\label{rsDmeq}
For $k\in\nat$ let $Z^{(k)}$ be a random variable on $\nat_0$ with distribution \eqref{rsDm}, i.e.\ for parameters $\theta\geq0$ and $\lambda\in(0,1)$ we have
$$P\{Z^{(k)}=n\}=p_k(\theta,\lambda;n)=\frac{1}{S(k,\theta,\lambda)}\,\frac{(\theta+\lambda n)^{n+k-1}}{n!}\,e^{-(\theta+\lambda n)},\quad{n\in\nat_0}.$$
Let $(Y_n)_{n\in\nat}$ be an i.i.d.\ sequence with a Borel distribution of parameter $\lambda$ and let $(N^{(m)})_{m\in\nat}$ be independent random variables with a Delaporte distribution \eqref{Delconv} of parameters $\theta,\lambda$ and $m\in\nat$. Then the distribution of $Z^{(k)}$ is representable as the compound randomly shifted Delaporte mixture
\begin{equation}\label{rsDmrep}
Z^{(k)}\eqd\sum_{n=1}^{V_k+N^{(k+V_k)}}Y_n,
\end{equation}
where $V_k$ is a random variable on $\{0,\ldots,k-1\}$ such that $V_k,N^{(k)},\ldots,N^{(2k-1)},(Y_n)_{n\in\nat}$ are independent and the distribution of $V_k$ can be written as
\begin{equation}\label{Vdistr}
P\{V_k=n\}=\frac{(1-\lambda)^{-1}}{S(k,\theta,\lambda)}\,q_k(n),\quad n=0,\ldots,k-1,
\end{equation}
where the quantities $q_k(n)$ are recursively given by $q_1(0)=1$ and 
\begin{equation}\label{qrec}
q_k(n)=\frac{(\theta+\lambda n)q_{k-1}(n)}{1-\lambda}+\frac{\lambda^2(k+n-2)q_{k-1}(n-1)}{(1-\lambda)^2},\quad n=0,\ldots,k-1
\end{equation}
with the convention $q_k(-1)=0=q_k(k)$ for every $k\in\nat$. 
\end{theorem}
The following corollary shows that we can avoid infinite summation in \eqref{Srec} due to \eqref{Vdistr} and \eqref{qrec}.
\begin{cor}
For $k \in \nat$, $\theta \ge 0$, and $\lambda \in (0,1)$ it holds that
\begin{equation}\label{Salt}
S(k,\theta,\lambda)=\frac1{1-\lambda}\sum_{n=0}^{k-1}q_k(n).
\end{equation}
\end{cor}
\begin{remark}
Note that for $k=0$ and $\theta>0$ the statement of Theorem \ref{rsDmeq} remains true by setting $V_0=0$ almost surely, since the Delaporte distribution $N^{(0)}$ with vanishing negative binomial part can be interpreted as a Poisson distribution of parameter $\theta>0$. Thus the distribution of   $Z^{(0)}$ is a GPD of Section 3.1.
\end{remark}
\begin{proof}[Proof of Theorem \ref{rsDmeq}]
Let $G$ be the pgf of a Borel distribution of parameter $\lambda\in(0,1)$ and for every $k\in\nat_0$ let $G_k$ be the pgf of \eqref{rsDm}. We already know that $G_0(z)=\exp(\theta(G(z)-1))$ is the pgf of a GPD. Further, note that
\begin{align*}
& \sum_{m=0}^\infty z^m P \left \{\sum_{n=1}^{V_k+N^{(k+V_k)}}Y_n = m \right\} = \sum_{m=0}^\infty z^m \sum_{\ell=0}^{k-1} P \left \{\sum_{n=1}^{\ell+N^{(k+\ell)}}Y_n = m \right\} P \{V_k = \ell \} \\
& \quad= \sum_{\ell=0}^{k-1} P \{V_k = \ell \} \sum_{m=0}^\infty z^m P \left \{\sum_{n=1}^{\ell+N^{(k+\ell)}}Y_n = m \right\} \\
& \quad=G_0(z) \sum_{\ell=0}^{k-1} P \{V_k = \ell \} G^\ell(z) \left (\frac{1-\lambda}{1-\lambda G(z)} \right)^{k+l} \\
& \quad = G_0(z) f_k(G(z)),
\end{align*}
where $f_k$ is the pgf of the corresponding mixture of shifted negative binomial distributions. Hence, to prove \eqref{rsDmrep} it remains to  show inductively that for $k\in\nat$ we have
\begin{equation}\label{Grep}
G_k(z)=G_0(z)f_k(G(z)),
\end{equation}
where by \eqref{Vdistr} and the above calculation
\begin{equation}\label{rsNBpgf}
f_k(z)=\frac{(1-\lambda)^{-1}}{S(k,\theta,\lambda)}\sum_{\ell=0}^{k-1}q_k(\ell)\left(\frac{1-\lambda}{1-\lambda z}\right)^{k+\ell}z^\ell.
\end{equation}
Note that for $k=1$ with $q_1(0)=1$ we have that $f_1(z)=\frac{1-\lambda}{1-\lambda z}$ is the pgf of a geometric distribution on $\nat_0$ in accordance with our statement. Differentiating \eqref{Grep} we get using \eqref{Gprime}
\begin{equation}\label{Gkprime}\begin{split}
G_k'(z) & =G_0'(z)f_k(G(z))+G_0(z)G'(z)f_k'(G(z))\\
& =G_0(z)\theta G'(z)f_k(G(z))+G_0(z)G'(z)f_k'(G(z))\\
& =G_0(z)\frac{G(z)}{z(1-\lambda G(z))}\,\big(\theta f_k(G(z))+f_k'(G(z))\big).
\end{split}\end{equation}
From \eqref{rsDm} we get using \eqref{Grep} and \eqref{Gkprime}
\begin{align*}
G_k(z) & =\sum_{n=0}^\infty z^np_k(\theta,\lambda;n)=\sum_{n=0}^\infty z^n\frac{S(k-1,\theta,\lambda)}{S(k,\theta,\lambda)}\,(\theta+\lambda n)\,p_{k-1}(\theta,\lambda;n)\\
& =\frac{S(k-1,\theta,\lambda)}{S(k,\theta,\lambda)}\,\big(\theta G_{k-1}(z)+\lambda z G_{k-1}'(z)\big)\\
& =\frac{S(k-1,\theta,\lambda)}{S(k,\theta,\lambda)}G_0(z)\Big(\theta f_{k-1}(G(z))+\frac{\lambda G(z)}{1-\lambda G(z)}\,\big(\theta f_{k-1}(G(z))+f_{k-1}'(G(z))\big)\Big)\\
& =\frac{S(k-1,\theta,\lambda)}{S(k,\theta,\lambda)}\,G_0(z)\,\frac{\theta f_{k-1}(G(z))+\lambda G(z)f_{k-1}'(G(z))}{1-\lambda G(z)}.
\end{align*}
Hence, in order to prove \eqref{Grep} it suffices to show
\begin{equation}\label{fkrec}
f_k(z)=\frac{S(k-1,\theta,\lambda)}{S(k,\theta,\lambda)}\,\frac{\theta f_{k-1}(z)+\lambda z\,f_{k-1}'(z)}{1-\lambda z}.
\end{equation}
This will follow from the recursion of $q_k$ defined in  \eqref{qrec}. Observe that
\begin{align*}
& \frac{S(k-1,\theta,\lambda)}{S(k,\theta,\lambda)}\,\frac{\theta f_{k-1}(z)+\lambda z\,f_{k-1}'(z)}{1-\lambda z}\\
& \quad=\frac{S(k-1,\theta,\lambda)}{S(k,\theta,\lambda)}\,\frac1{1-\lambda z}\left(\theta\frac{(1-\lambda)^{-1}}{S(k-1,\theta,\lambda)}\sum_{\ell=0}^{k-2}q_{k-1}(\ell)\left(\frac{1-\lambda}{1-\lambda z}\right)^{k-1+\ell}z^\ell\right.\\
& \quad+\lambda z\frac{(1-\lambda)^{-1}}{S(k-1,\theta,\lambda)}\sum_{\ell=0}^{k-2}\left[(k-1+\ell)q_{k-1}(\ell)\left(\frac{1-\lambda}{1-\lambda z}\right)^{k-2+\ell}z^\ell\frac{\lambda(1-\lambda)}{(1-\lambda z)^2}\right.\\
& \quad\qquad\qquad\qquad\qquad\qquad\left.\left.+\ell q_{k-1}(\ell)\left(\frac{1-\lambda}{1-\lambda z}\right)^{k-1+\ell}z^{\ell-1}\right]\right)\\
& \quad=\frac{(1-\lambda)^{-1}}{S(k,\theta,\lambda)}\left(\sum_{\ell=0}^{k-2}\frac{\theta q_{k-1}(\ell)}{1-\lambda}\left(\frac{1-\lambda}{1-\lambda z}\right)^{k+\ell}z^\ell\right.\\
& \quad \qquad \qquad\quad+\sum_{\ell=1}^{k-1}\frac{\lambda^2(k+\ell-2)q_{k-1}(\ell-1)}{(1-\lambda)^2}\left(\frac{1-\lambda}{1-\lambda z}\right)^{k+\ell}z^\ell\\
& \left.\quad\qquad \qquad\quad+\sum_{\ell=1}^{k-2}\frac{\lambda\ell q_{k-1}(\ell)}{1-\lambda}\left(\frac{1-\lambda}{1-\lambda z}\right)^{k+\ell}z^\ell\right)\\
& \quad= \frac{(1-\lambda)^{-1}}{S(k,\theta,\lambda)}\sum_{\ell=0}^{k-1}\left(\!\frac{(\theta+\lambda\ell)q_{k-1}(\ell)}{1-\lambda}+\frac{\lambda^2(k+\ell-2)q_{k-1}(\ell-1)}{(1-\lambda)^2}\!\right)\left(\frac{1-\lambda}{1-\lambda z}\right)^{k+\ell}\!z^\ell\\
& \quad = f_k(z),
\end{align*}
where the last equality follows from the recursive definition \eqref{qrec} of $q_k$
and \eqref{rsNBpgf}. This shows \eqref{fkrec} and concludes the proof.
\end{proof}
\begin{remark}
Note that for $k\in-\nat$ in general it is not possible to represent \eqref{rsDm} as a compound distribution with Borel summands. To demonstrate this fact, assume that for $k=-1$
$$p_{-1}(\theta,\lambda;n)=P\left\{\sum_{k=1}^NY_k=n\right\}\quad\text{ for all }n\in\nat_0,$$
where $(Y_n)_{n\in\nat}$ are i.i.d.\ random variables with a Borel distribution of parameter $\lambda\in(0,1)$ and $N$ is an independent random variable with values in $\nat_0$. Then, setting $C:=S(-1,\theta,\lambda)^{-1}>0$ we have
\begin{align*}
P\{N=0\} & =p_{-1}(\theta,\lambda;0)=C\cdot \theta^{-2}e^{-\theta},\\
P\{N=1\} & =P\{Y_1=1\}^{-1}p_{-1}(\theta,\lambda;1)=C\cdot (\theta+\lambda)^{-1}e^{-\theta}
\end{align*}
and hence
\begin{align*}
C\cdot\frac12\,e^{-(\theta+2\lambda)} & =p_{-1}(\theta,\lambda;1)=P\left\{\sum_{k=1}^NY_k=2\right\}\\
& =P\{Y_1=2\}P\{N=1\}+P\{Y_1=Y_2=1\}P\{N=2\}\\
& =C\cdot\frac{\lambda}{\theta+\lambda}e^{-(\theta+2\lambda)}+e^{-2\lambda}P\{N=2\}.
\end{align*}
Thus, for small values of $\theta$ we obtain the contradiction
$$P\{N=2\}=C\cdot\left(\frac12-\frac{\lambda}{\theta+\lambda}\right)e^{-\theta}<0.$$
For $k\in-\nat$ it remains an open question, whether under certain conditions it might be possible to represent \eqref{rsDm} as a compound distribution with Borel summands. 
\end{remark}

\subsection{Overview on the presented mixtures of the Borel distribution}

Let $(Y_n)_{n\in\nat}$ be an i.i.d.\ sequence of random variables with a Borel distribution of parameter $\lambda\in(0,1)$ and let $N$ be an independent random variable on $\nat_0$. As before $Z=\sum_{k=1}^NY_k$ denotes a random variable with the corresponding compound distribution. Table 1 gives an overview on the presented compound distributions, where $P(\theta)$, $G(\lambda)$ and  $N\!B(\lambda,m)$ denote the Poisson distribution with parameter $\theta>0$, the geometric distribution on $\nat_0$ with parameter $\lambda$ and the negative binomial distribution with parameters $\lambda$ and $m\geq2$. The independent random variables $V_k$ with distribution \eqref{Vdistr} and $N^{(m)}$ with Delaporte distribution $P(\theta)\ast N\!B(\lambda,m)$ are as in Theorem \ref{rsDmeq}.

\begin{table}\begin{tabular}{|c|l|c|c|}\hline
$\phantom{\Big|}N$ & $P\{Z=n\}$ & $\Exp[Z]$ & $\Var(Z)$ \\ \hline\hline
$\phantom{\Big|}N=1$ & $\frac{(\lambda n)^{n-1}}{n!}\,e^{-\lambda n}$ & $\frac{1}{1-\lambda}$ & $\frac{\lambda}{(1-\lambda)^3}$ \\ \hline
$\phantom{\Big|}N=m$ & $\frac{m(\lambda n)^{n-m}}{n(n-m)!}\,e^{-\lambda n}$ & $\frac{m}{1-\lambda}$ & $\frac{m\lambda}{(1-\lambda)^3}$ \\ \hline
$\phantom{\Big|}P(\theta)$ & $\frac{\theta(\theta+\lambda n)^{n-1}}{n!}\,e^{-(\theta+\lambda n)}$ & $\frac{\theta}{1-\lambda}$ & $\frac{\theta}{(1-\lambda)^3}$ \\ \hline
$\phantom{\Big|}P(\theta)\ast G(\lambda)$ & $\frac{(1-\lambda)(\theta+\lambda n)^{n}}{n!}\,e^{-(\theta+\lambda n)}$ & $\frac{\theta}{1-\lambda}+\frac{\lambda}{(1-\lambda)^2}$ & $\frac{\theta}{(1-\lambda)^3}+\frac{\lambda^2+\lambda}{(1-\lambda)^4}$ \\ \hline
$\phantom{\Big|}P(\theta)\ast N\!B(\lambda,m)$ & $\frac{(1-\lambda)^m(\theta+\lambda n+\lambda\,\alpha(m-1))^{n}}{n!}\,e^{-(\theta+\lambda n)}$ & $\frac{\theta}{1-\lambda}+\frac{m\lambda}{(1-\lambda)^2}$ & $\frac{\theta}{(1-\lambda)^3}+\frac{m(\lambda^2+\lambda)}{(1-\lambda)^4}$ \\ \hline
$\phantom{\Big|}N=V_k+N^{(k+V_k)}$ & $\frac{1}{S(k,\theta,\lambda)}\,\frac{(\theta+\lambda n)^{n+k-1}}{n!}\,e^{-(\theta+\lambda n)}$ & by \eqref{rsDmmom} & by \eqref{rsDmmom} \\ \hline
\end{tabular}
\vspace*{1ex}
\caption{\small Overview on the distributions (rows: Borel, Borel-Tanner, GPD, compound Bartlett, compound Delaporte, compound shifted Delaporte mixture) and their expectation and variance.}
\end{table}

Note that the Poisson, Bartlett and Delaporte distributions are infinitely divisible. The roots of the latter are known as Charlier distributions, see \cite{MB,WA} for details. Since the Borel distribution is a shifted compound Poisson distribution which follows from  \eqref{Borelpgf}, it is also infinitely divisible and thus it follows that the corresponding compound distributions with Borel summands of Sections 3.1--3.3 are infinitely divisible. It is an open question, whether the compound shifted Delaporte mixtures of Section 3.4 are infinitely divisible for $k\not\in\{0,1\}$.

We emphasize that the compound distributions of $Z$ in Table 1 all belong to the class of discrete general Lagrangian probability distributions also called Lagrange distributions of the first kind, i.e.\ we have $P\{Z=0\}=f(0)$ and
\begin{equation}\label{Ldfk}
P\{Z=n\}=\frac1{n!}\,\frac{d^{n-1}}{dz^{n-1}}\left\{g^n(z)f'(z)\right\}\big|_{z=0}\quad\text{ for }n\in\nat,
\end{equation}
where (in the simplest case) $f,g$ are pgf's of discrete distributions on $\nat_0$. For a comprehensive study of Lagrangian probability distributions we refer to the monograph \cite{CF}. In our model of total progeny in a Galton-Watson type branching process in Sections 2 and 3, $f$ is the pgf of the number $X_0$ of initial insurance claims or, equivalently, the pgf of the random number $N$ in Table 1, and $g(z)=\exp(\lambda(z-1))$ is the common pgf of the i.i.d.\ consequential claims, which here are assumed to have a Poisson distribution of parameter $\lambda\in(0,1)$. With these settings it follows readily from section 6.2 of \cite{CF} that \eqref{Ldfk} holds for all the distributions of $Z$ given in Table 1. This fact is well known for the distributions in the first four rows of Table 1, where due to $f(z)=z$ the Borel distribution in the first row is called a basic Lagrangian distribution, and due to $f(z)=z^m$ the Borel-Tanner distribution in the second row is called a delta Lagrangian distribution; cf.\ tables 2.1 and 2.2 in \cite{CF}. In Jain \cite{Jain} the compound Bartlett distribution is called a linear function Poisson distribution, which by Janardan \cite{Jan} is shown to be a weighted GPD, i.e.\ for a random variable $Z$ following a GPD \eqref{SD} we have the distribution
\begin{equation}\label{wGPD}
\frac{w(\theta,\lambda;n)\,p_{S\!D}(n)}{\Exp[w(\theta,\lambda;Z)]}\quad,\,n\in\nat_0,
\end{equation}
with nonnegative weights $w(\theta,\lambda;n)$ having positive and finite expectation in the denominator. In general, weighted Lagrangian distributions belong to the class of Lagrange distributions of the second kind given by
\begin{equation}\label{Ldsk}
\frac1{n!}\,\big(1-g'(1)\big)\,\frac{d^{n}}{dz^{n}}\left\{g^n(z)f(z)\right\}\big|_{z=0}\quad\text{ for }n\in\nat,
\end{equation}
provided that $g'(1)$ exists; see \cite{Jan} for details. Since by Theorem 2.1 of \cite{CF} the class of Lagrange distributions of the second kind belongs to those of the first kind, weighted Lagrangian distributions are general Lagrangian distributions as well. For the special weights $w(\theta,\lambda;n)=\theta+\lambda n$, clearly \eqref{wGPD} coincides with the compound Bartlett distribution $p_{SU}(n)$ in \eqref{SU} given in the fourth row of Table 1; cf.\ also \cite{Jan} or table 2.4 in \cite{CF}. The explicit forms of the Langrangian distributions in the last two rows of Table 1 seem to be new. Note that the distributions in \eqref{rsDm} for arbitrary $k\in\ganz$ can also be interpreted as weighted Lagrangian distributions with weights of the form $w_k(\theta,\lambda;n)=(\theta+\lambda n)^k$ in \eqref{wGPD}, since the normalizing constants in \eqref{rsDm} fulfill $S(k,\theta,\lambda)=\Exp[w_k(\theta,\lambda;Z)]$.

\section{recursive evaluation of related compound distributions}

We continue to give an actuarial interpretation of our models. Let $(U_n)_{n\in\nat}$ be an i.i.d.\ sequence of claim sizes with values in $\nat$ and with (known) probabilities $f(n)=P\{U_1=n\}$ for $n\in\nat$. The random number $Z$ of claims is assumed to be independent of $(U_n)_{n\in\nat}$ and to follow one of the distributions of Section 3, i.e.\ GPD, compound Bartlett, compound Delaporte or \eqref{rsDm}. Then the total claim size is again given by a compound sum
\begin{equation}\label{totcla}
T=\sum_{n=1}^{Z} U_n.
\end{equation}
If $Z$ follows a GPD, compound Bartlett distribution or, more generally a distribution of the form \eqref{rsDm} for some $k\in\ganz$, its distribution fulfills the recursive relation
\begin{equation}\label{GPDrec}
P\{Z=n\}=p(\theta,\lambda;n)=\left(a+\frac{b}{n}\right)p(\theta+\lambda,\lambda;n-1)\quad\text{ for }n\in\nat,
\end{equation}
where $a=\theta\lambda/(\theta+\lambda)$, $b=\theta^2/(\theta+\lambda)$ by \eqref{gpdrecur} in a case of a GPD and $a=\lambda$, $b=\theta$ by \eqref{Barmixrecur} in case of a compound Bartlett distributon with Borel summands. In the general case of a distribution of the form \eqref{rsDm} for some $k\in\ganz$ we additionally have to multiply the parameters $a$ and $b$ with a constant factor by \eqref{rsDmrecur}. Despite the fact that the first parameter $\theta$ changes to $\theta+\lambda$ on the right-hand side of \eqref{GPDrec} this relation is known to imply a simple recursion formula for the probabilities 
\begin{equation}\label{Tdense}
P\{T=n\}=q(\theta,\lambda;n)\quad\text{ for }n\in\nat_0
\end{equation}
of the aggregate claims in \eqref{totcla}. For details we refer to Panjer's classical result \cite{Pan} or its extensions in \cite{Schr1,Sun}; for an overview, e.g.\ see \cite{Dic,Schr2}. Originally these recursive methods were intended to decrease computational time but, due to ongrowing computer power, nowadays concurrent methods of numerical inversion of the pgf by FFT techniques with exponential tilting seem to be at least equally powerful; see \cite{EF}. Nevertheless, recursion formulas of Panjer type are easily programmable and thus provide a simple technique to calculate compound distributions. In case of a GPD the following result already appears in Theorem 5.1 of Ambagaspitiya and Balakrishnan \cite{AB} but without detailed proof. For an alternative method we refer to Goovaerts and Kaas \cite{GK}.
\begin{theorem}\label{Panrec}
For i.i.d.\ claim sizes $(U_n)_{n\in\nat}$ with distribution $f(n)=P\{U_1=n\}$, $n\in\nat$, and a claim number $Z$ fulfilling \eqref{GPDrec} the distribution \eqref{Tdense} of the total claim size \eqref{totcla} fulfills
\begin{align*}
q(\theta,\lambda;0) & =p(\theta,\lambda;0)\\
q(\theta,\lambda;n) & =\sum_{k=1}^n\left(a+\frac{bk}{n}\right)f(k)\,q(\theta+\lambda,\lambda;n-k) \quad\text{ for }n\in\nat.
\end{align*}
\end{theorem}
To calculate the distribution of the total claim size \eqref{Tdense} we can thus follow the recursive scheme
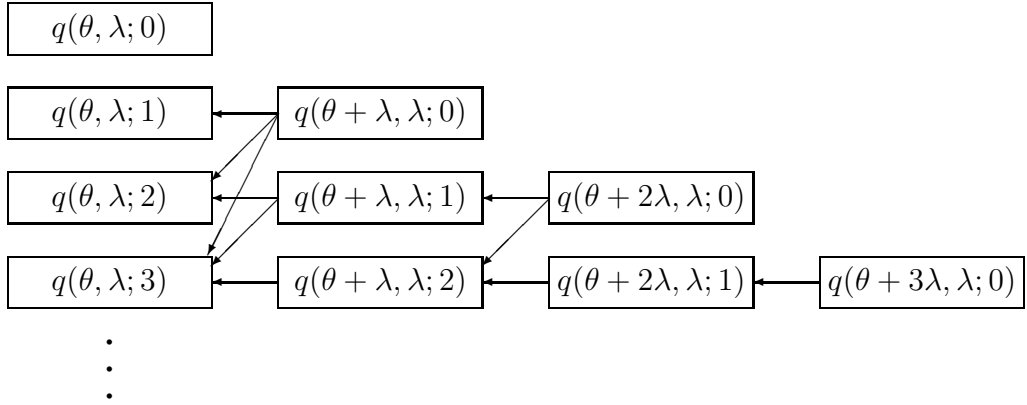
\begin{center}
\setlength{\unitlength}{.9cm}
\begin{picture}(16,6.5)
\put(0.5,1.5){\framebox(3,0.75){$q(\theta,\lambda;3)$}}
\put(0.5,2.75){\framebox(3,0.75){$q(\theta,\lambda;2)$}}
\put(0.5,4){\framebox(3,0.75){$q(\theta,\lambda;1)$}}
\put(0.5,5.25){\framebox(3,0.75){$q(\theta,\lambda;0)$}}
\put(4.5,1.5){\framebox(3,0.75){$q(\theta+\lambda,\lambda;2)$}}
\put(4.5,2.75){\framebox(3,0.75){$q(\theta+\lambda,\lambda;1)$}}
\put(4.5,4){\framebox(3,0.75){$q(\theta+\lambda,\lambda;0)$}}
\put(8.5,1.5){\framebox(3,0.75){$q(\theta+2\lambda,\lambda;1)$}}
\put(8.5,2.75){\framebox(3,0.75){$q(\theta+2\lambda,\lambda;0)$}}
\put(12.5,1.5){\framebox(3,0.75){$q(\theta+3\lambda,\lambda;0)$}}
\put(4.5,1.875){\vector(-1,0){1}}
\put(8.5,1.875){\vector(-1,0){1}}
\put(12.5,1.875){\vector(-1,0){1}}
\put(4.5,3.125){\vector(-1,0){1}}
\put(8.5,3.125){\vector(-1,0){1}}
\put(4.5,4.375){\vector(-1,0){1}}
\put(4.5,3.125){\vector(-1,-1){1}}
\put(4.5,4.375){\vector(-1,-1){1}}
\put(8.5,3.125){\vector(-1,-1){1}}
\put(4.5,4.375){\vector(-1,-2){1.05}}
\put(2,1){\circle*{0.08}}
\put(2,0.6){\circle*{0.08}}
\put(2,0.2){\circle*{0.08}}
\end{picture}
\end{center}
where for $m\in\nat_0$
$$q(\theta+m\lambda,\lambda;0)=p(\theta+m\lambda,\lambda;0)=\begin{cases}
e^{-(\theta+m\lambda)} & \text{for the GPD}\\
(1-\lambda)e^{-(\theta+m\lambda)} & \text{for the Bartlett mixture}\\
\frac{(\theta+m\lambda)^{k-1}}{S(k,\theta+m\lambda,\lambda)}\,e^{-(\theta+m\lambda)}& \text{for \eqref{rsDm} with $k\in\ganz$}
\end{cases}$$
by Theorems \ref{gpd}, \ref{Barmix} and \eqref{rsDm}, respectively.
\begin{proof}[Proof of Theorem \ref{Panrec}]
We follow a standard proof of Panjer's recursion, e.g.\ as given in Proposition 2.4 of \cite{AA} or Theorem 3.3.10 in \cite{Mik}. Since we assume $P\{U_1=0\}=0$, for $n=0$ we have
$$q(\theta,\lambda;0)=P\left\{\sum_{j=1}^{Z}U_j=0\right\}=p(\theta,\lambda;0).$$
Denote $S_m=\sum_{j=1}^mU_j$ for $m\in\nat$, then due to the i.i.d.\ nature of $(U_j)_{j\in\nat}$ for every $n\in\nat$ the conditional expectation $\Exp[a+b\frac{U_i}{n}\mid S_m=n]$ is independent of $i=1,\ldots,m$ and thus
$$\Exp\Big[a+b\,\frac{U_1}{n}\Big|S_m=n\Big]=\frac1m\sum_{i=1}^m\Exp\Big[a+b\,\frac{U_i}{n}\Big|S_m=n\Big]=\frac1m(ma+b)=a+\frac{b}{m}.$$
Hence for $n\in\nat$ we get by independence of $Z$, $(U_j)_{j\in\nat}$ and \eqref{GPDrec}
\begin{align*}
& q(\theta,\lambda;n)=\sum_{m=1}^\infty P\left\{\sum_{j=1}^mU_j=n,\,Z=m\right\}=\sum_{m=1}^nP\{S_m=n\}\,p(\theta,\lambda;m)\\
& \quad=\sum_{m=1}^nP\{S_m=n\}\left(a+\frac{b}{m}\right)\,p(\theta+\lambda,\lambda;m-1)\\
& \quad=\sum_{m=1}^n\Exp\left[\left(a+b\,\frac{U_1}{n}\right)\cdot 1_{\{S_m=n\}}\right]\,p(\theta+\lambda,\lambda;m-1)\\
& \quad=\sum_{m=1}^n\sum_{k=1}^{n-m+1}\Exp\left[\left(a+b\,\frac{k}{n}\right)\cdot 1_{\{S_m=n\}}\bigg|\,U_1=k\right]\,f(k)\,p(\theta+\lambda,\lambda;m-1)\\
& \quad=\sum_{m=1}^n\sum_{k=1}^{n-m+1}\left(a+b\,\frac{k}{n}\right)\,P\{S_{m-1}=n-k\}\,f(k)\,p(\theta+\lambda,\lambda;m-1)\\
& \quad=\sum_{k=1}^n\left(a+b\,\frac{k}{n}\right)\,f(k)\sum_{m=0}^{n-k}P\{S_{m}=n-k\}\,p(\theta+\lambda,\lambda;m)\\
& \quad=\sum_{k=1}^n\left(a+b\,\frac{k}{n}\right)\,f(k)\,q(\theta+\lambda,\lambda;n-k)
\end{align*}
concluding the proof.
\end{proof}
In case of the compound Delaporte distribution with Borel summands we need to introduce an additional parameter and may replace \eqref{GPDrec} by 
\begin{equation}\label{Delrec}
P\{Z=n\}=p(\theta,\lambda,m;n)=\sum_{i=1}^2\left(a_i+\frac{b_i}{n}\right)p(\theta+\lambda,\lambda,m+i-1;n-1)
\end{equation}
for all $n\in\nat$, where $a_1=a_2=0$ and $b_1=\theta + \lambda n$, $b_2=(m-1) \lambda/(1-\lambda)$ by Lemma \ref{Lem:Delmixrec} for the compound Delaporte distribution with Borel summands. Denoting the total claim size distribution by
\begin{equation}\label{TdenseDel}
P\{T=n\}=q(\theta,\lambda,m;n)\quad\text{ for }n\in\nat_0
\end{equation}
a similar result to Theorem \ref{Panrec} holds.
\begin{theorem}\label{PanrecDel}
For i.i.d.\ claim sizes $(U_n)_{n\in\nat}$ with distribution $f(n)=P\{U_1=n\}$, $n\in\nat$, and a claim number $Z$ fulfilling \eqref{Delrec} the distribution \eqref{TdenseDel} of the total claim size \eqref{totcla} fulfills
\begin{align*}
q(\theta,\lambda,m;0) & =p(\theta,\lambda,m;0)\\
q(\theta,\lambda,m;n) & =\sum_{i=1}^2\sum_{k=1}^n\left(a_i+\frac{b_ik}{n}\right)f(k)\,q(\theta+\lambda,\lambda,m+i-1;n-k) \quad\text{ for }n\in\nat.
\end{align*}
\end{theorem}
\begin{proof}
The assertion follows by the same line of arguments as given in the proof of Theorem \ref{Panrec} but using \eqref{Delrec} instead of \eqref{GPDrec}.
\end{proof}
To calculate the distribution of the total claim size \eqref{TdenseDel} we can thus follow the recursive scheme
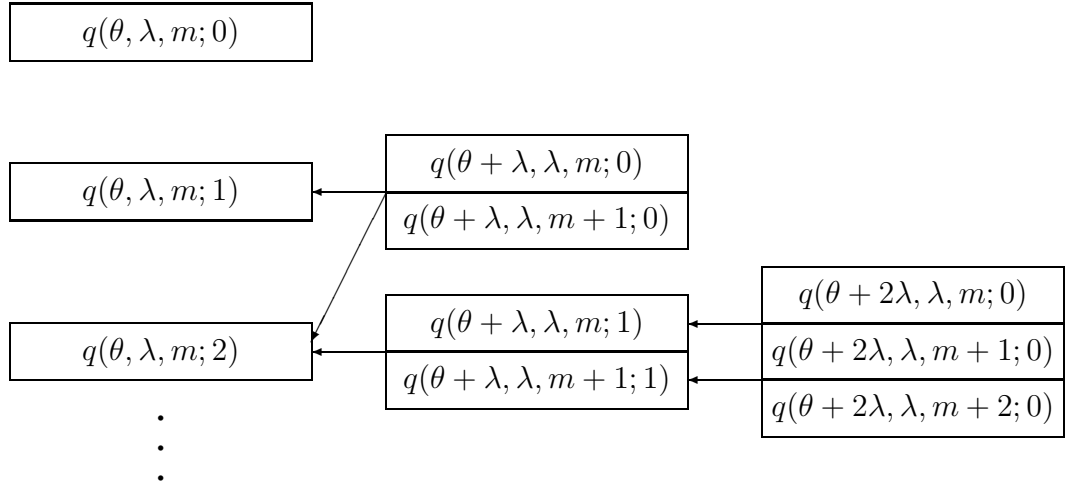
\begin{center}
\setlength{\unitlength}{1cm}
\begin{picture}(15,7)
\put(0.5,1.5){\framebox(4,0.75){$q(\theta,\lambda,m;2)$}}
\put(0.5,3.625){\framebox(4,0.75){$q(\theta,\lambda,m;1)$}}
\put(0.5,5.75){\framebox(4,0.75){$q(\theta,\lambda,m;0)$}}
\put(5.5,1.125){\framebox(4,0.75){$q(\theta+\lambda,\lambda,m+1;1)$}}
\put(5.5,1.875){\framebox(4,0.75){$q(\theta+\lambda,\lambda,m;1)$}}
\put(5.5,3.25){\framebox(4,0.75){$q(\theta+\lambda,\lambda,m+1;0)$}}
\put(5.5,4){\framebox(4,0.75){$q(\theta+\lambda,\lambda,m;0)$}}
\put(10.5,0.75){\framebox(4,0.75){$q(\theta+2\lambda,\lambda,m+2;0)$}}
\put(10.5,1.5){\framebox(4,0.75){$q(\theta+2\lambda,\lambda,m+1;0)$}}
\put(10.5,2.25){\framebox(4,0.75){$q(\theta+2\lambda,\lambda,m;0)$}}
\put(5.5,1.875){\vector(-1,0){1}}
\put(5.5,4){\vector(-1,0){1}}
\put(10.5,1.5){\vector(-1,0){1}}
\put(10.5,2.25){\vector(-1,0){1}}
\put(5.5,4){\vector(-1,-2){1}}
\put(2.5,1){\circle*{0.08}}
\put(2.5,0.6){\circle*{0.08}}
\put(2.5,0.2){\circle*{0.08}}
\end{picture}
\end{center}
where  for $k,\ell\in\nat_0$
$$q(\theta+k\lambda,\lambda,m+\ell;0)=p(\theta+k\lambda,\lambda,m+\ell;0)=(1-\lambda)^{m+\ell}e^{-(\theta+k\lambda)}$$
in case of the compound Delaporte distribution by Theorem \ref{Delmix}. Note that we do not have to calculate the compound Delaporte distribution with Borel summands \eqref{Delmixdens} explicitly, i.e.\ we can avoid to calculate Riordan's $\alpha$-symbols $\alpha^k(m-1)$ in \eqref{Delmixdens}. In particular, Theorem \ref{PanrecDel} can be applied to recursively evaluate the compound Delaporte distribution with Borel summands itself when taking $a_1=a_2=0$, $b_1=\theta + \lambda n$, $b_2=(m-1)\lambda/(1-\lambda)$ and $f(1)=1$, $f(n)=0$ for all $n\geq2$.
\begin{remark}
It is also possible to derive similar recursion formulas to Theorems \ref{Panrec} and \ref{PanrecDel} for i.i.d.\ claim sizes $(U_n)_{n\in\nat}$ with an absolutely continuous distribution with respect to Lebesgue measure on $\rr_+$. This leads to certain variants of Volterra integral equations of the second kind which can only be solved numerically. In case of the GPD this integral equation is derived in Theorem 4.1 of \cite{AB}. We renounce to present these integral equations, since in our actuarial context these are of minor practical importance.
\end{remark}

\appendix
\section*{Appendix}
\renewcommand{\thesection}{A}
\setcounter{equation}{0}
\setcounter{theorem}{0}

Our aim is to derive the Borel distribution from the characteristic equation \eqref{Borelpgf} of its pgf by probabilistic arguments and elementary combinatorial identities only. Equation \eqref{Borelpgf} can equivalently be written in terms of random variables as 
\begin{equation}\label{Borelrv}
Y\eqd 1+\sum_{k=1}^{N}Y_k,
\end{equation}
where $Y,Y_1,Y_2,\ldots$ are i.i.d.\ random variables with values in $\nat$ and $N$ is an independent random variable with a Poisson distribution of parameter $\lambda\in(0,1]$.
\begin{theorem}\label{Boreldensity}
For $\lambda\in(0,1]$ the distribution of $Y$ in \eqref{Borelrv} is uniquely given by the Borel distribution of parameter $\lambda$ with
\begin{equation}\label{BoreldenseA}
P\{Y=n\}=\frac{(\lambda n)^{n-1}}{n!}\,e^{-\lambda n}\quad\text{ for }n\in\nat.
\end{equation}
\end{theorem}
We will first prove a certain multinomial formula with variable frequencies which looks similar to a multinomial Abel identitiy; cf.\ section 1.6 in \cite{Rio}.
\begin{lemma}\label{multid}
For $n\in\nat$ and $k=1,\ldots,n$ we have
\begin{equation}\label{partial1}
\sum_{(n_1,\ldots,n_k)\in\nat^k\atop n_1+\cdots+n_k=n}\frac{n!}{n_1!\cdots n_k!\,k!}\prod_{\ell=1}^k\left(\frac{n_\ell}{n}\right)^{n_\ell-1}={n-1\choose k-1}.
\end{equation}
\end{lemma}
\begin{proof}
If $k=1$ then $n_1=n$ and \eqref{partial1} reads
$$\frac{n!}{n!\,1!}\left(\frac{n}{n}\right)^{n-1}=1={n-1\choose 0}.$$
If $k=n$ then $n_1=\cdots=n_k=1$ and \eqref{partial1} reads
$$\frac{n!}{1!\cdots 1!\,n!}\prod_{\ell=1}^n\left(\frac1n\right)^0=1={n-1\choose n-1}.$$
We will now show \eqref{partial1} by induction over $n\in\nat$. First, if $n=1$ then $k=1$ and the validity of \eqref{partial1} has been shown above. Now assume that \eqref{partial1} is true for all non-negative integers up to some $n\in\nat$ and all $k=1,\ldots,n$, then formula \eqref{partial1} for $n+1$ only needs to be deduced for $k=2,\ldots,n$. We get
\begin{align*}
& \sum_{(n_1,\ldots,n_k)\in\nat^k\atop n_1+\cdots+n_k=n+1}\frac{(n+1)!}{n_1!\cdots n_k!\,k!}\prod_{\ell=1}^k\left(\frac{n_\ell}{n+1}\right)^{n_\ell-1}\\
& \quad =\sum_{m=1}^{n+2-k}\frac1{m!}\left(\frac{m}{n+1}\right)^{m-1}\sum_{(n_1,\ldots,n_{k-1})\in\nat^{k-1}\atop n_1+\cdots+n_{k-1}=n+1-m}\frac{(n+1)!}{n_1!\cdots n_{k-1}!\,k!}\prod_{\ell=1}^{k-1}\left(\frac{n_\ell}{n+1}\right)^{n_\ell-1}\\
& \quad =\sum_{m=1}^{n+2-k}\frac{(n+1)!}{m!\,(n+1-m)!}\left(\frac{m}{n+1}\right)^{m-1}\left(\frac{n+1-m}{n+1}\right)^{n+2-m-k}\frac1k{n-m\choose k-2}\\
& \quad =\sum_{m=1}^{n+2-k}\frac{{n+1\choose m}{n-m\choose k-2}}{{n-(k-1)\choose m-1}}\,\frac1k{n-(k-1)\choose m-1}\left(\frac{m}{n+1}\right)^{m-1}\left(\frac{n+1-m}{n+1}\right)^{n+2-m-k}\\
& \quad ={n\choose k-1}\frac{k-1}{k\,(n+1)^{n-k}}\sum_{m=0}^{n+1-k}{n+1-k\choose m}(m+1)^{m-1}(n-m)^{n-m-k},
\end{align*}
since by elementary calculations we have
$$\frac{{n+1\choose m}{n-m\choose k-2}}{{n-(k-1)\choose m-1}}={n\choose k-1}\frac{(n+1)(k-1)}{m(n+1-m)}.$$
The remaining sum on the right-hand side is an Abel sum $A_{n+1-k}(1,k-1,-1,-1)$ as defined in section 1.5 of \cite{Rio} and its closed-form solution is given by 
$$\sum_{m=0}^{n+1-k}{n+1-k\choose m}(m+1)^{m-1}(n-m)^{n-m-k}=\frac{k(n+1)^{n-k}}{k-1};$$
cf.\ \cite{Rio}, page 20. This proves \eqref{partial1} and concludes the proof.
\end{proof}
\begin{proof}[Proof of Theorem \ref{Boreldensity}]
We will prove the assertion inductively. First, we observe $P\{Y=1\}=P\{N=0\}=e^{-\lambda}$ which coincides with \eqref{BoreldenseA} for $n=1$. Now assume that formula \eqref{BoreldenseA} is true for  $P\{Y=1\},\ldots,P\{Y=n\}$ with some $n\in\nat$. Then we get using Lemma \ref{multid}
\begin{align*}
& P\{Y=n+1\}=P\left\{\sum_{k=1}^{N}Y_k=n\right\}=\sum_{k=1}^nP\{N=k\}\sum_{(n_1,\ldots,n_k)\in\nat^k\atop n_1+\cdots+n_k=n}\prod_{\ell=1}^kP\{Y=n_\ell\}\\
& \quad =\sum_{k=1}^n\frac{\lambda^k}{k!}\,e^{-\lambda}\sum_{(n_1,\ldots,n_k)\in\nat^k\atop n_1+\cdots+n_k=n}\prod_{\ell=1}^k\frac{n_\ell^{n_\ell}}{n_\ell!}\,\lambda^{n_\ell-1}e^{-n_\ell \lambda}\\
& \quad =\frac{\lambda^ne^{-(n+1)\lambda}}{(n+1)!}\sum_{k=1}^n\frac{(n+1)!}{k!}\sum_{(n_1,\ldots,n_k)\in\nat^k\atop n_1+\cdots+n_k=n}\prod_{\ell=1}^k\frac{n_\ell^{n_\ell}}{n_\ell!}\\
& \quad =\frac{\lambda^ne^{-(n+1)\lambda}}{(n+1)!}\sum_{k=1}^nn^{n-k}\sum_{(n_1,\ldots,n_k)\in\nat^k\atop n_1+\cdots+n_k=n}\frac{(n+1)!}{n_1!\cdots n_k!\,k!}\prod_{\ell=1}^k\left(\frac{n_\ell}{n}\right)^{n_\ell-1}\\
& \quad=\frac{\lambda^ne^{-(n+1)\lambda}}{(n+1)!}\,(n+1)\sum_{k=1}^nn^{n-k}{n-1\choose k-1}\\
& \quad =\frac{\lambda^ne^{-(n+1)\lambda}}{(n+1)!}\,(n+1)\sum_{k=0}^{n-1}{n-1\choose k}n^{(n-1)-k}=\frac{\lambda^ne^{-(n+1)\lambda}}{(n+1)!}(n+1)^n
\end{align*}
which proves the assertion. 
\end{proof}
The above elementary approach to the Borel distribution can also be used to derive its convolution powers as follows.
\begin{theorem}
For fixed $m\in\nat$ let $Y^{(m)}=Y_1+\cdots+Y_m$, where $Y_1,\ldots,Y_m$ are i.i.d.\ random variables with a Borel distribution of parameter $\lambda\in(0,1]$. Then $Y^{(m)}$ has a Borel-Tanner distribution with
\begin{equation*}
P\{Y^{(m)}=n\}=\frac{m\,(\lambda n)^{n-m}}{n\,(n-m)!}\,e^{-\lambda n}\quad\text{ for }n\in\nat\text{ with }n\geq m.
\end{equation*}
\end{theorem}
\begin{proof}
For $Y\eqd Y_1$ we get from \eqref{BoreldenseA}
\begin{align*}
& P\{Y^{(m)}=n\}=P\left\{\sum_{k=1}^{m}Y_k=n\right\}=\sum_{(n_1,\ldots,n_m)\in\nat^m\atop n_1+\cdots+n_m=n}\prod_{\ell=1}^mP\{Y=n_\ell\}\\
& \quad =\sum_{(n_1,\ldots,n_m)\in\nat^m\atop n_1+\cdots+n_m=n}\prod_{\ell=1}^m\frac{n_\ell^{n_\ell-1}}{n_\ell!}\,\lambda^{n_\ell-1}e^{-\lambda n_\ell}\\
& \quad =\frac{m! (\lambda n)^{n-m}}{n!}\,e^{-\lambda n}\sum_{(n_1,\ldots,n_m)\in\nat^m\atop n_1+\cdots+n_m=n}\frac{n!}{n_1!\cdots n_m!\,m!}\prod_{\ell=1}^m\left(\frac{n_\ell}{n}\right)^{n_\ell-1}\\
& \quad =\frac{m! (\lambda n)^{n-m}}{n!}\,e^{-\lambda n}{n-1\choose m-1}=\frac{m\,(\lambda n)^{n-m}}{n\,(n-m)!}\,e^{-\lambda n},
\end{align*}
where the first equality in the last line follows from Lemma \ref{multid}.
\end{proof}

\end{document}